\documentclass{amsart}
\usepackage[T1]{fontenc}
\usepackage[latin9]{inputenc}
\usepackage{verbatim}
\usepackage{mathtools}
\usepackage{amsthm}
\usepackage{amssymb}

\makeatletter
\numberwithin{equation}{section}
\numberwithin{figure}{section}
\theoremstyle{plain}
\newtheorem{thm}{\protect\theoremname}
  \theoremstyle{plain}
  \newtheorem{question}[thm]{\protect\questionname}
  \theoremstyle{definition}
  \newtheorem{defn}[thm]{\protect\definitionname}
  \theoremstyle{plain}
  \newtheorem{prop}[thm]{\protect\propositionname}
  \theoremstyle{definition}
  \newtheorem{example}[thm]{\protect\examplename}
  \theoremstyle{remark}
  \newtheorem{claim}[thm]{\protect\claimname}
  \theoremstyle{remark}
  \newtheorem*{acknowledgement*}{\protect\acknowledgementname}

\usepackage{tikz}
\usetikzlibrary{shapes.geometric}

\usepackage{ifpdf} % part of the hyperref bundle
 \IfFileExists{lmodern.sty}{\usepackage{lmodern}}{}

\usepackage{mathtools}

\makeatother

  \providecommand{\acknowledgementname}{Acknowledgement}
  \providecommand{\claimname}{Claim}
  \providecommand{\definitionname}{Definition}
  \providecommand{\examplename}{Example}
  \providecommand{\propositionname}{Proposition}
  \providecommand{\questionname}{Question}
\providecommand{\theoremname}{Theorem}

\begin{document}
\global\long\def\rank{\operatorname{rank}}

\global\long\def\girth{\operatorname{girth}}

\global\long\def\diam{\operatorname{diam}}

\title{Commuting Graphs of Boundedly Generated Semigroups}

\author{Tomer Bauer and Be'eri Greenfeld}

\address{Department of Mathematics, Bar-Ilan University, Ramat Gan 52900,
Israel}
\begin{abstract}
Araújo, Kinyon and Konieczny \cite{AKK/11} pose several problems
concerning the construction of arbitrary commuting graphs of semigroups.

We observe that every star-free graph is the commuting graph of some
semigroup. Consequently, we suggest modifications for some of the
original problems, generalized to the context of families of semigroups
with a bounded number of generators, and pose related problems.

We construct monomial semigroups with a bounded number of generators,
whose commuting graphs have an arbitrary clique number. In contrast
to that, we show that the diameter of the commuting graphs of semigroups
in a wider class (containing the class of nilpotent semigroups), is
bounded by the minimal number of generators plus two.

We also address a problem concerning knit degree.
\end{abstract}

\thanks{The first author is partially supported by the Bar-Ilan University
President's Doctoral Fellowships of Excellence Program.}
\maketitle

\section{Introduction}

We denote the center of a semigroup $S$ by $Z(S)$. The \textit{commuting
graph} of $S$, denoted by $\Gamma(S)$, is the simple graph whose
vertex set is $S-Z(S)$, where two vertices are connected by an edge
if their corresponding elements in $S$ commute. For example, if $S$
is a commutative semigroup, then $\Gamma(S)$ is the empty graph.
A study of these graphs in relation with certain algebraic structures
can be found, for instance in \cite{AAM/06,AGHM/04,AKK/11,GP/10,Segev/99}.

Araújo, Kinyon and Konieczny \cite{AKK/11} proved that for every
natural number $n\geq2$, there exists a semigroup whose commuting
graph has diameter $n$. They also pose the following question:
\begin{question}[\cite{AKK/11}, Section 6, (4)]
\label{thm:graph-proporties}Can every number $n\geq3$ be the clique
number (girth, chromatic number) of the commuting graph of a semigroup?
\end{question}

All semigroups constructed in this note are assumed to have a zero
element, denoted $0$. We call a semigroup \emph{monomial} if it has
a presentation with only monomial relations (i.e.\ relations of the
form $w=0$).

We now classify which graphs are realizable as commuting graphs of
semigroups, thereby answering Question \ref{thm:graph-proporties}.
\begin{defn}
A vertex in a graph is called a \emph{star} if it is connected to
any other vertex. A graph is called \emph{star-free} if none of its
vertices is a star.

It is obvious that any commuting graph is star-free. The following
result shows that the converse is true as well.
\end{defn}

\begin{prop}
\label{prop:generic-construction}Let $\Gamma=\left(V,E\right)$ be
a star-free graph. Then there exists a (monomial) semigroup $S$ such
that $\Gamma(S)=\Gamma$. Moreover, if $\Gamma$ has $n$ vertices,
then there exists a (monomial) semigroup $S$, with $\Gamma(S)=\Gamma$,
such that $S$ is generated by $n$ elements and its order is $n^{2}+n+1-\left|E\right|$.
\end{prop}

\begin{proof}
Label the vertices of $\Gamma$ to generate a free semigroup with
zero. Consider its quotient semigroup
\[
S=\left\langle v\in V\left|\begin{gathered}v_{1}v_{2}v_{3}=0\ \forall v_{1},v_{2},v_{3}\\
vu=uv\ \forall\left(u,v\right)\in E
\end{gathered}
\right.\right\rangle .
\]
It is clear that the center $Z(S)$ consists of all elements of the
form $vu$. It is now evident why $\Gamma(S)=\Gamma$. Also, an easy
calculation shows that if $\Gamma$ has $n$ vertices then this semigroup
has precisely $n^{2}+n+1-\left|E\right|$ elements.

Note we could set the relations $uv=vu=0$ whenever $\left(u,v\right)\in E$,
resulting in a monomial semigroup with the same commuting graph.
\end{proof}
Let $\rank(\Gamma)$ denote the minimal number of generators for a
semigroup whose commuting graph is $\Gamma$. Note that the construction
in Proposition \ref{prop:generic-construction} only shows that $\rank(\Gamma)\leq\left|V\right|$.
For a semigroup $S$, we call the minimal number of generators the
\emph{rank of $S$}.

Thus, one can modify Question \ref{thm:graph-proporties} and ask
for families of semigroups generated by a bounded number of generators,
yet having arbitrary graph-theoretic invariants. We exhibit such families
with arbitrary clique numbers (and leave open the same question about
girth). 

For a certain class of finite semigroups (which contains the class
of nilpotent semigroups), we show that the diameter of the commuting
graph is bounded by a linear function of the rank.

In Section \ref{sec:knit-degree-3} we provide a counterexample to
another problem from \cite{AKK/11} concerning the notion of knit
degree (defined in the sequel).

\section{Commuting Graphs with Arbitrary Clique Number\label{sec:clique-number}}

Denote the clique number of a graph $\Gamma$ by $\omega(\Gamma)$.
In view of the previous section it is evident that there are semigroups
whose commuting graphs have arbitrary clique number. However, the
semigroups that arise from Proposition \ref{prop:generic-construction}
have rank at least the prescribed clique number, thus unbounded.

To exhibit a family of boundedly generated semigroups with arbitrary
clique number, we first introduce a family of monomial semigroups
$S_{n}$ having arbitrary odd clique number $2n-1$, all of which
are generated by two elements.

Commuting graphs with even clique number are then constructed by the
simple operations of null union of semigroups and join of graphs:
\begin{defn}
Let $S,T$ be two semigroups with zero. Their \emph{null union} $S\bullet T$
is the semigroup $S\cup T$ after identifying together their zeros.
For every pair of elements $s\in S$ and $t\in T$ their product in
$S\bullet T$ is defined to be $st=ts=0$.
\end{defn}

\begin{defn}
The \emph{join} of two simple graphs, denoted by $\vee$, is the graph
union with all the edges that connect the vertices of the first graph
with the vertices of the second graph.
\end{defn}

\begin{prop}
\label{prop:arbitrary-clique-number}Let $n$ be a natural number.
Then

\begin{enumerate}
\item \label{enu:clique-number-odd}The semigroup 
\[
S_{n}=\left<a,b\ |\ a^{n+1}=b^{2}=0,\ aba=ba^{i}b=0\;\forall i\ge1\right>
\]
 satisfies $\omega(\Gamma(S_{n}))=2n-1$.
\item \label{enu:clique-number-even}The semigroup $S_{n}\bullet S_{1}$
satisfies $\omega(\Gamma(S_{n}\bullet S_{1}))=2n$.
\end{enumerate}
\end{prop}

\begin{example}
The graph $\Gamma(S_{2})$ is 
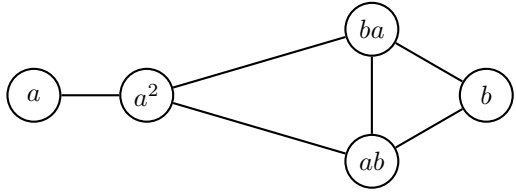
\begin{center}
\begin{tikzpicture}
  \tikzstyle{vertex}=[circle,draw,thick,fill=white,minimum size=20pt, inner sep=0pt]

  \node(triangle)[regular polygon, regular polygon sides=3, minimum height=2cm, rotate=-90] at (5,0) {};

  \node[vertex] (1) at (triangle.corner 1) {$b$};
  \node[vertex] (2) at (triangle.corner 2) {$ba$};
  \node[vertex] (3) at (triangle.corner 3) {$ab$};
  \node[vertex] (4) at (0,0) {$a$};
  \node[vertex] (5) at (1.5,0) {$a^{2}$};
  \path[draw,thick]
    % The triangle clique
    (1) edge node {} (2)
    (1) edge node {} (3)
    (2) edge node {} (3)
    % The powers of $a$ clique
    (4) edge node {} (5)
    % The rest of the edges
    (5) edge node {} (2)
    (5) edge node {} (3)
    ;
\end{tikzpicture}
\end{center}

and the graph $\Gamma(S_{3})$ is 
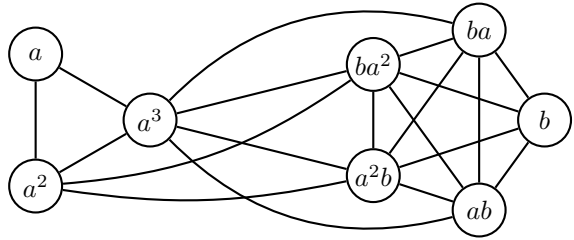
\begin{center}
\begin{tikzpicture}
  \tikzstyle{vertex}=[circle,draw,thick,fill=white,minimum size=20pt, inner sep=0pt]

  \node(triangle)[regular polygon, regular polygon sides=3, minimum height=2cm, rotate=30] at (0,0) {};
  \node(pentagon)[regular polygon, regular polygon sides=5, minimum height=2.5cm, rotate=-90] at (5,0) {};

  \node[vertex] (1) at (pentagon.corner 1) {$b$};
  \node[vertex] (2) at (pentagon.corner 2) {$ba$};
  \node[vertex] (3) at (pentagon.corner 3) {$ba^{2}$};
  \node[vertex] (4) at (pentagon.corner 4) {$a^{2}b$};
  \node[vertex] (5) at (pentagon.corner 5) {$ab$};
  \node[vertex] (6) at (triangle.corner 1) {$a$};
  \node[vertex] (7) at (triangle.corner 2) {$a^{2}$};
  \node[vertex] (8) at (triangle.corner 3) {$a^{3}$};
  \path[draw,thick]
    % The pentagon clique
    (1) edge node {} (2)
    (1) edge node {} (3)
    (1) edge node {} (4)
    (1) edge node {} (5)
    (2) edge node {} (3)
    (2) edge node {} (4)
    (2) edge node {} (5)
    (3) edge node {} (4)
    (3) edge node {} (5)
    (4) edge node {} (5)
    % The powers of $a$ clique
    (6) edge node {} (7)
    (6) edge node {} (8)
    (7) edge node {} (8)
    % The rest of the edges
    (8) edge [bend left=30] node {} (2)
    (7) edge [bend right=15] node {} (3)
    (8) edge node {} (3)
    (7) edge [bend right=10] node {} (4)
    (8) edge node {} (4)
    (8) edge [bend right=30] node {} (5)
    ;
\end{tikzpicture}
\end{center}
\end{example}

\begin{proof}[Proof of \eqref{enu:clique-number-odd}]
Clearly $S_{n}=\{0,b,a,a^{2},\cdots,a^{n},ab,a^{2}b,\cdots,a^{n}b,ba,ba^{2},\cdots,ba^{n}\}$,
so $S_{n}$ has $3n+2$ elements. We calculate $Z(S_{n})=\{0,a^{n}b,ba^{n}\}$,
so $|\Gamma(S_{n})|=3n-1$. Take 
\[
C=\{a^{i}b\}_{1\leq i\leq n-1}\cup\{ba^{j}\}_{1\leq j\leq n-1}\cup\{b\}
\]
which is clearly a clique of size $2n-1$, and so $\omega(\Gamma(S_{n}))\ge2n-1$.

Now, let $C'$ be a clique of $\Gamma(S_{n})$. If $b\in C'$, then
$a^{i}\notin C'$ for all $i$ so $|C'|\leq2n-1$. If $a^{i}\in C'$
with the minimal such $i$, then $b,a,\cdots,a^{i-1},a^{j}b,ba^{j}\notin C'$
for all $j\leq n-i$, which are together $2n-i$ vertices, so 
\[
|C'|\leq(3n-1)-[1+(i-1)+2(n-i)]=n+i-1\leq2n-1
\]
If both $b\notin C'$ and $a^{i}\notin C'$ for all $i$, then $\left|C'\right|\le2n-2$.
This completes the proof of the first claim that $\omega(\Gamma(S_{n}))=2n-1$.
\end{proof}
\begin{proof}[Proof of \eqref{enu:clique-number-even}]
In order to construct commuting graphs with arbitrary even clique
number, we shall need the following properties of join of graphs and
null union of semigroups.

\begin{enumerate}
\item Let $G,H$ be two graphs. Then $\omega(G\vee H)=\omega(G)+\omega(H)$.
\item Let $S,T$ be two semigroups with zero. Then $\Gamma(S\bullet T)\cong\Gamma(S)\vee\Gamma(T)$
(naturally), as $Z(S\bullet T)=Z(S)\bullet Z(T)$.
\item Let $S,T$ be two semigroups with zero. Then $\omega(\Gamma(S\bullet T))=\omega(\Gamma(S))+\omega(\Gamma(T))$.
\end{enumerate}
Thus, to exhibit boundedly generated semigroups whose commuting graph
has arbitrary even clique number, it is enough to consider
\[
S_{1}\coloneqq\left<a,b|a^{2}=b^{2}=0,\!aba=bab=0\right>=\{0,a,b,ab,ba\}.
\]
Clearly we have that $Z(S_{1})=\{0,ab,ba\}$ and $\Gamma(S_{1})\cong\overline{K}_{2}$,
the edgeless graph with two vertices.

By the above properties and the fact that $\omega(\Gamma(S_{1}))=1$
one obtains that $\omega(\Gamma(S_{n}\bullet S_{1}))=2n$ and that
$S_{n}\bullet S_{1}$ can be generated by $4$ elements.

Notice that the commuting graph of the semigroup $S'_{n}\coloneqq\underbrace{S_{1}\bullet\cdots\bullet S_{1}}_{n}$
has clique number $n$. However, the semigroup $S'_{n}$ can not be
generated by less than $2n$ elements, which is obviously unbounded.
\end{proof}

\section{Diameter of Commuting Graphs}

Recall that the diameter of a (connected) graph is the maximal distance
between two vertices. In \cite{AKK/11}, the authors introduce constructions
of semigroups with connected commuting graphs having arbitrary diameter.
The constructions there have unbounded rank that grows linearly with
the diameter. By Proposition \ref{prop:generic-construction} it is
again possible to obtain such examples, now with rank equal to the
number of vertices (hence bigger than the prescribed diameter).

Proposition \ref{prop:arbitrary-clique-number} exhibits commuting
graphs of monomial semigroups with at most $4$ generators having
arbitrary clique number. For a wider family of semigroups, we show
next that the diameter of their commuting graphs is effectively bounded
by the rank of the semigroups.

We now show that there is a tight connection between the diameter
and the rank of a semigroup in a wide class of semigroups. For a semigroup
$S$ and $m\in\mathbb{N}$ denote 
\[
S^{m}\coloneqq\left\{ a_{1}a_{2}\cdots a_{m}|a_{i}\in S\ \forall i\right\} \subseteq S
\]
In particular we have $S^{m}\subseteq S^{m-1}\subseteq\dots\subseteq S$.
\begin{prop}
\label{prop:Diameter}Let $S$ be a semigroup generated by $d$ elements.
Suppose it has a non-central ideal $I$ such that $IS,SI\subseteq Z(S)$.
If $\Gamma(S)$ is connected, then its diameter $D$ satisfies $D\leq d+2$.
\end{prop}

\begin{proof}
Pick $u\in I\setminus Z(S)$. We claim that for every $v\in S^{2}$,
we have that $uv=vu$. Indeed, write $v=t_{1}t_{2}$, and compute
\[
uv=ut_{1}t_{2}=t_{2}ut_{1}=t_{1}t_{2}u=vu
\]
since $ut_{1}\in IS$ and $t_{2}u\in SI$ are central.

Let $x,y\in\Gamma(S)$ be two vertices. If $x,y\in S^{2}$ then a
path of length (at most) $2$ can be found between them, as $u$ commutes
with both of them. Suppose only one of $x,y$, say $x$, is not contained
in $S^{2}$ (but $y\in S^{2}$). Note that $\left|S\setminus S^{2}\right|\le d$,
because $S$ is generated by $d$ elements. Then by connectivity a
shortest path can be found between $x$ and some vertex $z\in S^{2}$.
As the path is minimal, its length does not exceed $d$, which is
the number of generators. We can now concatenate this path to $z-u-y$,
resulting in a path of length at most $d+2$, connecting $x$ and
$y$.

We now assume that both $x,y\notin S^{2}$. We can find shortest paths
$\gamma_{x},\gamma_{y}$ from $x,y$ to some $x',y'\in S^{2}$, respectively
and denote the lengths of these paths by $d_{x},d_{y}$. As before,
minimality implies that $\gamma_{x}$ consists of precisely $d_{x}$
vertices from $S\setminus S^{2}$ (and likewise $\gamma_{y}$ consists
of $d_{y}$ such vertices). But $d_{x}+d_{y}\leq d$, for otherwise
the pigeonhole principle would imply that some vertex from $S\setminus S^{2}$
appears in both $\gamma_{x},\gamma_{y}$, so a path between $x$ and
$y$ can be found of length at most $d-1$. Hence the path $x-\dots-x'-u-y'-\dots-y$
obtained by concatenating $\gamma_{x},x'-u-y',\gamma_{y}$ has length
at most $d+2$.
\end{proof}
Let $S$ be a non-commutative nilpotent semigroup of nilpotency degree
$c$, and let $m$ be the minimal exponent for which $S^{m}\subseteq Z(S)$.
Then a possible ideal satisfying the premise of the previous Proposition
is $I=S^{m-1}$. The nilpotency of $S$ guarantees $m\le c$ because
$S^{c}=\left\{ 0\right\} \subseteq Z(S)$. As mentioned, if $S$ is
commutative (e.g.\ $c\le2$), then its commuting graph is empty and
the claim holds trivially.

\section{A Semigroup with Knit Degree 3\label{sec:knit-degree-3}}

Another question that was posed in \cite{AKK/11} is related to the
notion of knit degree:
\begin{defn}[\cite{AKK/11}]
Let $S$ be a semigroup. A path $a_{1}-a_{2}-\dots-a_{m}$ in $\Gamma(S)$
is called a \emph{left path} (abbreviated in \cite{AKK/11} as $l$-path)
if $a_{1}\ne a_{m}$ and $a_{1}a_{i}=a_{m}a_{i}$ for every $1\le i\le m$.
If $\Gamma(S)$ contains a left path, then the \emph{knit degree}
of $S$ is defined to be the length of a shortest left path in $\Gamma(S)$.
\end{defn}

For $n=2$ and every $n\ge4$ a band with knit degree $n$ was constructed
in \cite{AKK/11}. In \cite[Section~6(1)]{AKK/11} it was guessed
that a semigroup with knit degree $3$ does not exist.%
\begin{comment}
Can we construct a band with knit degree $3$?
\end{comment}

We now provide a counterexample.%
\begin{comment}
If we add the relations $x_{i}^{2}=x_{i}$ and $x_{i}x_{j}x_{i}x_{j}=x_{i}x_{j}$
will we get a band with knit degree $3$?
\end{comment}
\begin{example}
The following semigroup has knit degree $3$:
\[
S=\left\langle x_{1},x_{2},x_{3},x_{4}|R\right\rangle 
\]
with the relations
\[
R=\left\{ \begin{array}{c}
x_{i}x_{j}x_{k}=0\ \forall i,j,k\\
x_{1}^{2}=x_{4}x_{1}\\
x_{4}^{2}=x_{1}x_{4}\\
x_{1}x_{2}=x_{4}x_{2}=x_{1}x_{3}=x_{4}x_{3}=x_{2}x_{1}=x_{2}x_{3}=x_{3}x_{2}=x_{3}x_{4}=0
\end{array}\right\} 
\]

One may verify that the graph $\Gamma(S)$ is a path graph on four
vertices: 
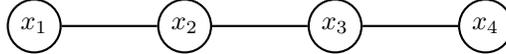
\begin{center}
\begin{tikzpicture}
  \tikzstyle{vertex}=[circle,draw,thick,fill=white,minimum size=20pt, inner sep=0pt]

  \node[vertex] (1) at (0,0) {$x_{1}$};
  \node[vertex] (2) at (2,0) {$x_{2}$};
  \node[vertex] (3) at (4,0) {$x_{3}$};
  \node[vertex] (4) at (6,0) {$x_{4}$};
  \path[draw,thick]
    (1) edge node {} (2)
    (2) edge node {} (3)
    (3) edge node {} (4)
    ;
\end{tikzpicture}
\end{center}The relations (except for $x_{i}x_{j}x_{k}=0$) ensure that $x_{1}-x_{2}-x_{3}-x_{4}$
is a left path.

We are left to show that a shorter left path does not exist in the
graph. There is an automorphism transposing $x_{1}\leftrightarrow x_{4}$
and $x_{2}\leftrightarrow x_{3}$, so it suffices to check that $x_{1}-x_{2}$,
$x_{2}-x_{3}$ and $x_{1}-x_{2}-x_{3}$ are not left paths, which
is an easy exercise.
\end{example}

\section{Further Problems}
\begin{question}
Can $\Gamma(S)$ be connected with arbitrarily large diameter for
$S$ with bounded rank?
\end{question}

We would like to remark that in general, there is no uniform bound
on the rank of finite graphs. In fact, we have:
\begin{claim}
Let $G$ be a $(2n-2)$-regular graph on $2n$ vertices. Then $\rank\left(G\right)=2n$.
\end{claim}

\begin{proof}
By the generic construction in Proposition \ref{prop:generic-construction},
we have that $\rank\left(G\right)\le2n$. To show that $\rank\left(G\right)=2n$,
consider any subset $X$ of at most $2n-1$ vertices of $G$. Then
$X$ lacks some vertex $v$, so there exists a vertex $u$ in $G$
(the one not connected to $v$) which is connected to all vertices
of $X$ except $u$ itself. Let $S$ be a semigroup such that $\Gamma(S)=G$.
Suppose that the elements of $X$ plus some (perhaps none) elements
of $Z(S)$ generate $S$. Then $u$ must lie in $Z(S)$. We conclude
that $\rank\left(G\right)=2n$.
\end{proof}
Following that, we restate another question that was posed in \cite{AKK/11}.
\begin{question}
Is there a family of graphs with bounded rank and unbounded girth?
\end{question}

We remark that for the class of semigroups satisfying the condition
in Proposition \ref{prop:Diameter}, the answer to the above question
is negative. 

It is natural to specialize and ask about $C_{n}$, the cyclic graph
with $n$ vertices.
\begin{question}
\label{conj:rank-3}What is the rank of $C_{n}$?
\end{question}

For example, here is an argument showing that the rank of $C_{4}$
is $4$. Clearly every generating subset must contain at least $2$
non-central generators, say $x_{1}$ and $x_{2}$. They must lie on
non-adjacent vertices of $C_{4}$. If the rank were $2$, then every
other vertex would have corresponded to a central element (as it commutes
with both $x_{1}$ and $x_{2}$), a contradiction. Hence, there must
be at least $3$ generators, but then the third generator again commutes
with all generators, so $4$ generators are needed. A similar reasoning
shows that the rank of $C_{5}$ is at least $3$.
\begin{acknowledgement*}
We are grateful to Stuart Margolis, Michael Schein, Uzi Vishne and
two anonymous referees for helpful comments.
\end{acknowledgement*}

\bibliographystyle{plain}
\bibliography{commuting_graphs}

\end{document}